\documentclass[12pt]{amsart}
\usepackage{amscd}

\numberwithin{equation}{section}

\newtheorem{thm}{Theorem}[section]

\newtheorem{lm}[thm]{Lemma}
\newtheorem{cor}[thm]{Corollary}
\newtheorem{prop}[thm]{Proposition}

\theoremstyle{definition}

\newtheorem{qn}[thm]{Question}

\newtheorem{remark}[thm]{Remark}

\newtheorem{dfn}[thm]{Definition}
\newtheorem{claim}[thm]{Claim}

\newtheorem{defn-thm}[thm]{Definition-Theorem}
\newtheorem{conjecture}[thm]{Conjecture}

\usepackage{amsmath}
\usepackage{amsthm}
\usepackage{commath}
\usepackage{amssymb}
\begin{document}
\title{Harmonic projections in negative curvature}
\author{Ognjen To\v{s}i\'{c}}\address{
    Mathematical Institute\\
    University of Oxford\\
    United Kingdom
}
 
\begin{abstract}
    In this paper we construct harmonic maps that are at a bounded distance from nearest-point retractions to convex sets, in negatively curved manifolds. Specifically, given a quasidisk $Q$ in hyperbolic space, we construct a harmonic map to the hyperbolic plane that corresponds to the nearest-point retraction to the convex hull of $Q$. If $M$ is a pinched Hadamard manifold so that its isometry group acts with cobounded orbits, and if $S$ is a set in the boundary at infinity of $M$, with the property that all elements of its orbit under the isometry group of $M$ have dimension less than $\frac{n-1}{2}$, we show that the nearest-point retraction to the convex hull of $S$ is a bounded distance away from some harmonic map.
\end{abstract}
\maketitle
\section{Introduction}
Yau conjectured in \cite[Question 38]{yau-problems} that any simply connected, complete K\"ahler manifold with sectional curvature at most $-1$ admits a holomorphic map onto the disk.
A natural analogue of this question for general (not necessarily K\"ahler) Riemannian manifolds is as follows.
\begin{qn}\label{qn:main} Does any pinched Hadamard manifold admit a non-constant harmonic map to the hyperbolic plane $\mathbb{H}^2$?
\end{qn}
A manifold is pinched Hadamard if it is simply connected, complete and with sectional curvature bounded between two negative constants.
We give a partial positive answer to this question. Our basic strategy is as follows. 
\begin{enumerate}
    \item We start with some quasi-isometric embedding $\iota:\mathbb{H}^2\to M$. This defines a quasicircle $S$ in the boundary at infinity of $M$. A modification of the nearest-point retraction onto the convex hull of $S$ gives a map $r:M\to\mathbb{H}^2$.
    \item We deform the map $r$ to a harmonic map.
\end{enumerate} 
We conjecture that this strategy works in general. 
\begin{conjecture}\label{conj:nearest-point-retractions}
    Let $M$ be a pinched Hadamard manifold, and let $\iota:\mathbb{H}^2\to M$ be a quasi-isometric embedding. There exists a harmonic map $h:M\to\mathbb{H}^2$ such that $\sup_{x\in \mathbb{H}^2}\mathrm{dist}(x, h\circ\iota(x))<\infty$.
\end{conjecture}
It was shown in \cite{benoist-hulin-pinched-hadamard} by Benoist and Hulin that any quasi-isometry between pinched Hadamard manifolds is at a bounded distance from a harmonic map, building on the work of Markovi\'c in \cite{m-schoen}. Note that the nearest-point retraction onto convex subsets is not a quasi-isometry, so the nearest-point retraction is outside the scope of their work.
\par We solve Conjecture \ref{conj:nearest-point-retractions} for all hyperbolic spaces $\mathbb{H}^n$.  
\begin{thm}\label{thm:main-hyp-space}
    Let $\iota:\mathbb{H}^2\to\mathbb{H}^n$ be a quasi-isometric embedding. Then there exists a harmonic map $h:\mathbb{H}^n\to\mathbb{H}^2$ such that 
    \begin{align*}
        \sup_{x\in\mathbb{H}^2}\mathrm{dist}(x, h(\iota(x)))<\infty.
    \end{align*}
\end{thm}
Note that $\iota$ as in Theorem \ref{thm:main-hyp-space} defines a quasi-circle $S$ in the boundary at infinity $\partial_\infty \mathbb{H}^n$. An essential ingredient in the proof of Theorem \ref{thm:main-hyp-space} is the existence of a constant $\beta<n-1$ such that for any isometry $\gamma$ of $\mathbb{H}^n$, the Minkowski dimension of $\gamma S$ is at most $\beta$, which essentially follows from the work of Gehring \cite{gehring}. More precisely, we use the fact that the invariant upper Minkowski dimension of $S$ is less than $n-1$ (for an exact definition see \S\ref{section:setting}).
\begin{thm}\label{thm:main-hyp-space-full}
    Let $S$ be a set in the sphere at infinity of $\mathbb{H}^n$ with invariant upper Minkowski dimension less than $n-1$. Then there exists a harmonic map $h:\mathbb{H}^n\to\mathbb{H}^n$ that is a bounded distance away from the nearest-point retraction to the convex hull of $S$.
\end{thm}
\par In the general case of a pinched Hadamard $n$-manifold $M$, we are only able to deal with nearest-point retractions onto convex hulls of subsets of the boundary of dimension less than $\frac{n-1}{2}$. Given a metric space $X$, we call a subset $S\subseteq X$ cobounded if for some $C>0$, the $C$-neighborhood of $S$ is $X$.
\begin{thm}\label{thm:main-hadamard}
    Let $M$ be a pinched Hadamard manifold of dimension $n$, such that the isometry group of $M$ has cobounded orbits. Let $S\subseteq\partial_\infty M$ be a closed set in the boundary at infinity of $M$, with the invariant upper Minkowski dimension less than $\frac{n-1}{2}$. Then there exists a harmonic map $h:M\to M$ at a bounded distance from the nearest-point retraction to the convex hull of $S$.
\end{thm}
In particular, when $M$ is the universal cover of a closed negatively curved manifold, the isometry group of $M$ has cobounded orbits.
\par In the Theorem \ref{thm:main-hadamard} above, we look for harmonic maps $M\to M$. From the proofs it will be clear that we could also construct harmonic maps $M\to Y$ when $Y\hookrightarrow M$ is quasi-isometrically embedded pinched Hadamard manifold with a similar condition on the invariant upper Minkowski dimension of the image of the boundary at infinity of $Y$.
\subsection{Outline}
The reason we are able to prove the stronger Theorem \ref{thm:main-hyp-space-full} for hyperbolic spaces is that in this case we have a precise estimate on the heat kernel. Apart from this, the proofs of Theorems \ref{thm:main-hyp-space-full} and \ref{thm:main-hadamard} are very similar, so the entire paper apart from \S\ref{section:proof} deals with the more general setting of Theorem \ref{thm:main-hadamard}. Hence let $S$ be a set in $\partial_\infty M$, and let $K$ be the convex hull of $S$. We let $r:M\to K$ be the nearest-point retraction. For the purposes of this outline, assume that $M$ has sectional curvature at most $-1$.
\par The proof consists of four steps.
\begin{enumerate}
    \item We construct a smooth map $\tilde{r}:M\to M$ that is at a bounded distance from $r$, so that its derivative and Hessian have the property $\norm{\nabla_x \tilde{r}}, \norm{H(\tilde{r})_x}\leq C e^{-\mathrm{dist}(x, K)}$, for some constant $C$. This $\tilde{r}$ is the result of a construction of Benoist and Hulin in \cite[\S 2.2]{benoist-hulin-pinched-hadamard}. Their exact statements do not apply here since $\tilde{r}$ is not a quasi-isometry, and moreover we need slightly stronger conclusions than they do. We summarize their construction in \S\ref{section:mollify-retraction}, and explain how it applies to $r$.
    \item Let $N_C(K)$ be the $C$-neighborhood of $K$, for some $C$ large. We show that assuming that the integral $\int_{N_C(K)} G(x,y)d\mathrm{vol}(y)$ of the Green's function $G(\cdot,\cdot)$ is bounded in $x$, there exists a bounded map $\Phi:M\to\mathbb{R}$ so that 
    \begin{align*}
        \Delta\Phi>\norm{\tau(\tilde{r})},
    \end{align*}
    where $\tau(\tilde{r})$ denotes the tension field of $\tilde{r}$ (the exact definition will be given in \S\ref{section:setting}).
    This is the content of Lemma \ref{lm:greens-estimate-implies-soln}. We use the assumption that the integral of the Green's function is bounded to construct $\Phi$ on $N_C(K)$. On $M\setminus N_C(K)$, we construct $\Phi$ as a suitable function of the distance $\mathrm{dist}(x, \tilde{r}(x))$.
    \item Denote the ball centered at $x$ of radius $d$ by $B(x, d)$. We construct harmonic maps $h_d:B(x, d)\to M$ that agree with $\tilde{r}$ on $\partial B(x,d)$. Then by an estimate of Schoen and Yau \cite{schoen-yau-laplacian-distance} on the Laplacian of the distance between smooth maps, we get
    \begin{align*}
        \Delta\left(\mathrm{dist}(h_d(x), \tilde{r}(x))+\Phi\right)>0.
    \end{align*}
    By the maximum principle we see that $\mathrm{dist}(h_d(x), \tilde{r}(x))\leq 2\sup_x\abs{\Phi(x)}$. Since this bound is uniform in $d$, a compactness argument shows that we can take a limit of $h_d$ as $d\to\infty$ to get a harmonic map that is at a bounded distance from $\tilde{r}$. This argument is in the proof of Corollary \ref{cor:finish-from-green}. This step essentially appears in the work of Donnelly \cite[Lemma 3.1]{donelly}.
    \item To finish the proof of Theorems \ref{thm:main-hyp-space-full} and \ref{thm:main-hadamard}, we only need to verify that $\int_{N_C(K)} G(x,y)d\mathrm{vol}(y)$ is bounded in $x$ for arbitrarily large $C$. We first show the bound
    \begin{align*}
            \mathrm{vol}(N_C(K)\cap B(x, \rho))\leq C' \exp\left(\left(\overline{\dim} S+\varepsilon\right)\rho\right),
        \end{align*}
        for all $x\in M, \varepsilon>0$, and some constant $C'$ that does not depend on $x$ or $\rho$.
    This estimate is shown in Lemma \ref{lm:low-dim-exp-growth}. We then use some classical estimates on the heat kernel $H:M\times M\times[0,\infty)\to \mathbb{R}$ (see Proposition \ref{prop:kernel-estimates}) and the fact that $G(x,y)=\int_0^\infty H(x,y,t)dt$. The computations combining these estimates are in \S\ref{subsection:main-hyp-space-full} and \S\ref{subsection:main-hadamard}.
\end{enumerate} 
\par The assumptions that $\overline{\dim} S<n-1$ and $\overline{\dim}S<\frac{n-1}{2}$ from Theorems \ref{thm:main-hyp-space-full} and \ref{thm:main-hadamard}, respectively, are only used when applying Lemma \ref{lm:low-dim-exp-growth} to obtain an estimate on the integral of the Green's function.
\par The proof of Theorem \ref{thm:main-hyp-space} also follows the outline above. The only difference is that we need to construct an initial map $\mathbb{H}^n\to\mathbb{H}^2$, and show that quasicircles have invariant upper Minkowski dimension less than $n-1$. All of this is done in \S\ref{subseciton:main-hyp-space}. 
\subsection{Acknowledgements}
I would like to thank Vladimir Markovi\'c for introducing me to Question \ref{qn:main} and for his continued advice and support while working on this project. I would also like to thank the anonymous referee for numerous comments that have improved the clarity of the paper.
\section{Preliminaries and notation}\label{section:setting}
Let $M$ be a pinched Hadamard manifold, that is a simply connected complete Riemannian manifold of dimension $n$ with sectional curvatures $K_M$ with $-b^2\leq K_M\leq -a^2$, for some fixed constants $0<a\leq b$. We assume that the isometry group of $M$ has cobounded orbits.  Recall that the group action $G$ on $X$ has cobounded orbits if for any point $x\in X$, there is a constant $C>0$ such that the $C$-neighborhood of the orbit $G\cdot x$ is $X$. 
\par These will be standing assumptions throughout the paper. In particular this holds whenever $M$ is the universal cover of some closed negatively curved manifold.
\par We denote by $\mathrm{dist}(\cdot, \cdot)$ the path metric on $M$, and by $[x,y]$ the geodesic segment connecting $x$ to $y$. We denote by $B(x, r)$ the ball centered at $x$ of radius $r$ in $M$, and by $N_d(Y)=\bigcup_{y\in Y} B(y, d)$ the $d$-neighborhood of the set $Y\subseteq X$. For a set $S\subseteq M$, we denote by $\mathrm{CH}(S)$ its convex hull, that is the intersection of all convex sets containing $S$.
\par We write $f\lesssim g$ when there exists a constant $C>0$ such that $f\leq Cg$. When it is not clear from context, we will specify what $C$ is allowed to depend on. We write $f\gtrsim g$ for $g\lesssim f$ and $f\approx g$ for $f\lesssim g\lesssim f$.
\subsection{Harmonic maps, Green's function and the heat kernel}
For a smooth map $h:X\to Y$ between Riemannian manifolds, we denote by $\nabla h$ its derivative, and by $H(h)$ its Hessian. Note that $\nabla h$ is a $h^*TY$-valued 1-form on $X$, and that $H(h)$ is a $h^*TY$-valued symmetric bilinear form on $X$.
\begin{dfn}
    For a smooth map $h:X\to Y$ between Riemannian manifolds, we define its tension field to be $\tau(h)=\mathrm{tr} H(h)$. The function $h$ is harmonic if $\tau(h)=0$. When $N=\mathbb{R}$, we denote $\Delta h=\tau(h)$.
\end{dfn}
\par We now recall the definitions of the heat kernel and Green's function. For more detailed information, the reader can consult the book by Grigor'yan \cite[Chapters 7, 11]{Grigoryan2012HeatKA}.
\begin{dfn}
    The heat kernel $H:\mathbb{R}\times M\times M\to\mathbb{R}$ is the unique smooth function such that for any smooth compactly supported function $f_0:M\to\mathbb{R}$, the function  
    \begin{align*}
        f_t(x)=\int_M H(t, x, y) f_0(y) d\mathrm{vol}(y)
    \end{align*} 
    is the solution to $\frac{\partial f_t}{\partial t}=\Delta f_t$, and has the property that $f_t\to f_0$ as $t\to 0$.
\end{dfn}
\begin{dfn}
    The Green's function is defined as \begin{align}
        G(x, y)=\int_0^\infty H(t, x, y)dt,\label{eq:green-conv}
    \end{align}
    whenever the right-hand side converges.
\end{dfn}
By \cite[Theorem 13.17]{Grigoryan2012HeatKA}, $G$ is the fundamental solution to the Laplace's equation whenever it is finite. It is well-known that Green's functions exist on all complete noncompact manifolds without boundary (e.g. \cite{greens-fn-exist}). We will in particular show in \S\ref{subsection:heat-kernel-estimates} that in our setting, the integral in (\ref{eq:green-conv}) converges. 
\subsection{Visual metrics and upper invariant Minkowski dimension}
\par Denote by $\partial_\infty M$ the boundary at infinity of $M$, that is the set of geodesic rays in $M$ up to the equivalence relation of having finite Hausdorff distance (for a more detailed account of the theory of boundaries of negatively curved spaces, the reader may wish to consult \cite{Kapovich2002BoundariesOH}). We set $\overline{M}=M\cup\partial_\infty M$, and extend the notation $\mathrm{CH}(S)$ and $[x, y]$ for $S\subseteq\overline{M}$ and $x, y\in\overline{M}$.
\par We equip $\partial_\infty M$ with the family of visual metrics $\mathrm{dist}_x^\mathrm{vis}(\cdot, \cdot)$ indexed by $x\in M$, given by
\begin{align*}
    \mathrm{dist}_x^\mathrm{vis}(y, z)\approx e^{-a\mathrm{dist}(x, [y, z])}.
\end{align*}
\begin{remark} Note that for general Gromov hyperbolic metric spaces, the visual metrics can only be defined as $\mathrm{dist}_x(y,z)\approx e^{-\kappa\mathrm{dist}(x, [y,z])}$, for some $\kappa>0$ small enough. However since $M$ is a $\mathrm{CAT}(-a^2)$ space, such a metric exists whenever $0<\kappa\leq a$, \cite[\S 2.4]{bourdon1993actions}. \end{remark}
\par The appropriate notion of dimension we will use for subsets of $\partial_\infty M$ is defined below. For a subset $S$ of some metric space $(X, d)$, we denote by $N_d(S, \varepsilon)$ the smallest number of $\varepsilon$-balls needed to cover $S$.
\begin{dfn}\label{dfn:minkowski-moving-bp}
    If $M$ is a pinched Hadamard manifold, for $S\subseteq \partial_\infty M$, the invariant upper Minkowski dimension of $S$, denoted $\overline{\dim} S$, is the infimum of all $d\geq 0$ with the property that there exists a constant $C$ such that  
    \begin{align*}
        N_{\mathrm{dist}_x^\mathrm{vis}}(S, \varepsilon)\leq C\varepsilon^{-d},
    \end{align*}
    for all $x\in M$ and $\varepsilon>0$. 
\end{dfn}
If we fix some arbitrary base point $o\in M$, and write $\mathrm{dist}^\mathrm{vis}=\mathrm{dist}^\mathrm{vis}_o$, the Definition \ref{dfn:minkowski-moving-bp} is equivalent to the definition below.
\begin{dfn}
    For a set $S\subseteq\partial_\infty M$, the upper-invariant Minkowski dimension is the infimum of all $d$ such that there exists a consant $C>0$ with the property that 
    \begin{align*}
        N_{\mathrm{dist}^\mathrm{vis}}(S, \varepsilon)\leq C\varepsilon^{-d}\text{ for all }\gamma\in\mathrm{Isom}(M)\text{ and }\varepsilon>0.
    \end{align*}
\end{dfn}
\section{Deforming the nearest-point retraction to a smooth map}\label{section:mollify-retraction}
In this section we deform the nearest-point retraction $r:M\to K$ to a convex set $K$ to a smooth map $\tilde{r}$ with $\sup_{x\in M}\mathrm{dist}(r(x), \tilde{r}(x))<\infty$, so that 
\begin{gather*}
    \norm{\nabla \tilde{r}}\lesssim e^{-a\mathrm{dist}(\cdot, K)},\\
    \norm{H(\tilde{r})(X, X)}\lesssim e^{-a\mathrm{dist}(\cdot, K)} \norm{X}^2.
\end{gather*}
In \S\ref{subsection:lipschitz}, we describe how to modify any Lipschitz map $f:X\to Y$ to a smooth map where the first two derivatives at $x\in X$ are controlled by the local Lipschitz constant of $f$ near $x$. This is essentially a restatement of the results of \cite[\S 2.2]{benoist-hulin-pinched-hadamard} suitable for our purposes. In \S\ref{subsection:local-lipschitz} we show that the local Lipschitz constant of the nearest-point retraction decays exponentially with the distance from the convex set.
\subsection{Deforming Lipschitz maps to smooth maps}\label{subsection:lipschitz}
We do this by the methods of Benoist and Hulin in \cite[\S 2.2]{benoist-hulin-pinched-hadamard}. We collect their results as Lemma \ref{lm:benoist-hulin}. 
\begin{lm}\label{lm:benoist-hulin}
    Let $f:X\to Y$ be a Lipschitz map between pinched Hadamard manifolds $X$ and $Y$. Then there exists a smooth map $\tilde{f}:X\to Y$ at a bounded distance from $f$ and a polynomial $P$ with non-negative coefficients and $P(0)=0$ such that whenever $x\in X$ and $f$ is $L$-Lipschitz in a neighborhood of $x$, we have 
    \begin{align*}
        \norm{\nabla_x \tilde{f}}\leq P(L)\text{ and }\norm{H(\tilde{f})_x}\leq P(L).
    \end{align*}
\end{lm}
\begin{proof}
    This is the result of \cite[Lemma 2.8]{benoist-hulin-pinched-hadamard} and a slight strengthening of \cite[Lemma 2.7]{benoist-hulin-pinched-hadamard}. We state these results as propositions below, and combine them as in \cite[Proof of Proposition 2.4. Second step]{benoist-hulin-pinched-hadamard}.
    \par We say that a subset $S\subset X$ is $r$-separated if for all $x, y\in S$, we have $\mathrm{dist}(x, y)\geq r$.
    \begin{prop}[Lemma 2.8 in \cite{benoist-hulin-pinched-hadamard}]\label{prop:decompose}
        There exist constants $r_0>0$ and $N_0\in\mathbb{Z}_{>0}$ such that for each $r<r_0$, any $r$-separated subset of $X$ can be decomposed into at most $N_0$ disjoint subsets, each of which is $4r$-separated.
    \end{prop}
    \begin{prop}[Strengthening of Lemma 2.7 in \cite{benoist-hulin-pinched-hadamard}]\label{prop:mollify-locally}
        Let $g:X\to Y$ be a map between pinched Hadamard manifolds. Then for all $r>0$ small enough, there exists a family of maps $g_{r, x}:X\to Y$ indexed by $x\in X$, such that
        \begin{align*}
            g_{r, x}(z)&=g(x)\text{ when }\mathrm{dist}(x, z)\leq\frac{r}{2},\\
            g_{r, x}&=g\text{ on }X\setminus B(x, r).
        \end{align*} 
        Moreover 
        \begin{align*}
            \mathrm{Lip}(g_{r, x}|_{B(x, r)})\lesssim \mathrm{Lip}(g|_{B(x, r)}),
        \end{align*}
        and if $g$ is $C^2$ in some neighborhood of $x$, then so is $g_{r, x}$ with 
        \begin{align*}
            \norm{H(g_{r, x})_z}\lesssim \mathrm{Lip}(g|_{B(x, r)})+\mathrm{Lip}(g|_{B(x, r)})^2+\norm{H(g)_z},
        \end{align*}
        for each $z$ in that neighborhood, where the implied constant depends on $r$.
    \end{prop}
    \begin{proof}
        We use the coordinates given by the following Proposition to construct $g_{r, x}$.
        \begin{prop}[Lemma 2.6 in \cite{benoist-hulin-pinched-hadamard}]\label{prop:almost-linear-coordinates}
            There exist constants $r_0>0$ and $c_0>1$ such that for any $y\in Y$, there exists a chart $\Phi_y:B(y, r_0)\to U_y\subseteq\mathbb{R}^{\dim Y}$ such that $\Phi_y(y)=0$ and 
            \begin{align*}
                \norm{\nabla\Phi_y}, \norm{\nabla \Phi_y^{-1}}, \norm{H(\Phi_y)}, \norm{H(\Phi_y^{-1})}\leq c_0.
            \end{align*}
            Here $U_y\subseteq\mathbb{R}^{\dim Y}$ is given the standard Euclidean metric. In particular, we have for $r<r_0$,
            \begin{align*}
                B\left(0, \frac{r}{c_0}\right)\subseteq\Phi_y(B(y, r))\subseteq B(0, c_0r).
            \end{align*}
        \end{prop}
        Let $\chi:\mathbb{R}\to [0,1]$ be a smooth function with $\chi|_{[-\frac{1}{2}, \frac{1}{2}]}=0$, and $\chi|_{\mathbb{R}\setminus[-1, 1]}=1$. Write $\chi_r(x)=\chi(\frac{x}{r})$, and let $\Phi_\cdot$ be the coordinates given by Proposition \ref{prop:almost-linear-coordinates}. We define 
        \begin{align*}
            g_{r, x}(z)&=\left\lbrace\begin{matrix}
                g(x) & \text{if }\mathrm{dist}(x, z)\leq \frac{r}{2}\\
                \Phi_x^{-1}\left(\chi_r(\mathrm{dist}(x, z))\Phi_x(g(z))\right) & \text{if }\frac{r}{2}\leq\mathrm{dist}(x, z)\leq r\\
                g(z)&\text{otherwise}
            \end{matrix}\right..
        \end{align*}
        Note that this is the exact same function as in \cite[Lemma 2.7]{benoist-hulin-pinched-hadamard}, and is well-defined when $c_0^2r\mathrm{Lip}(g)<r_0$. In this proof, denote by $d(z)=\mathrm{dist}(x, z)$, for ease of notation. We have for $\frac{r}{2}\leq d(z)\leq r$, 
        \begin{align*}
            \nabla_z g_{r, x}&=\nabla_{\chi_r(d(z))\Phi_x(g(z))}\Phi_x^{-1} \left(\chi_r'(d(z))\left(\nabla_z d\right) \Phi_x(g(z))+\chi_r(d(z))\nabla_{g(z)}\Phi_x \nabla_z g\right),
        \end{align*}
        so 
        \begin{align*}
            \norm{\nabla_z g_{r, x}}&\lesssim \norm{\Phi_x(g(z))}+\norm{\nabla_z g}\lesssim \norm{\nabla_z g}.
        \end{align*}
        Taking one more derivative, we see that 

        \begin{align*}
            \norm{H(g_{r, x})_z}\lesssim\norm{\nabla_z g}^2+\norm{\nabla_z g}+\norm{H(g)_z}.
        \end{align*}
        The result now follows from $\norm{\nabla_z g}\lesssim \mathrm{Lip}(g|_{B(x, r)})$. 
    \end{proof}
    The rest of the proof is completely analogous as \cite[Proof of Proposition 2.4. Second step]{benoist-hulin-pinched-hadamard}. Let $r>0$ be small enough to be chosen later. Let $X_0$ be a maximal $\frac{r}{2}$-separated subset of $X$. By Proposition \ref{prop:decompose} we can write $X_0=X_1\cup X_2\cup\cdots\cup X_{N_0}$ where each $X_i$ is $2r$-separated. Define $f_0=f$, and set 
    \begin{align*}
        f_i(z)&=\left\lbrace\begin{matrix} (f_{i-1})_{r, x}(z) & \text{if }z\in B(x, r)\text{ for some }x\in X_i\\
        f_{i-1}(z) & \text{otherwise} \end{matrix}\right.
    \end{align*}
    Set $\tilde{f}=f_{N_0}$. Then since each point $z\in X$ is in $B(x, \frac{r}{2})$ for some $x\in X_0$ (by maximality of $X_0$), some $f_i$ is locally constant near $x$ and hence $f_{N_0}$ is smooth at $x$. 
    \par By Proposition \ref{prop:mollify-locally}, whenever $f$ has Lipschitz constant at most $L$ near $x\in X$, we have 
    \begin{align*}
        \norm{\nabla_x \tilde{f}}\lesssim L,\\
        \norm{H(\tilde{f})_x}\leq \sum_{i=0}^{N_0} P^{\circ i}(L),
    \end{align*}
    where $P(z)=\Lambda(z+z^2)$ for some large enough constant $\Lambda>0$. The result follows.
\end{proof}
\subsection{Local Lipschitz constant of the nearest-point retraction}\label{subsection:local-lipschitz}
We now show that the local Lipschitz constant of the nearest-point retraction to $K$ decays exponentially with the distance from $K$. This is a basic result in $\mathrm{CAT}(-a^2)$ geometry, and is probably not new.
\begin{prop}\label{prop:nearest-point-lipschitz}
    Let $K$ be a closed convex subset of $M$, and let $r:M\to K$ be the nearest-point retraction. Then its restriction $r:M\setminus N_s(K)\to K$ has Lipschitz constant at most $Ce^{-as}$, where we equip $M\setminus N_s(K)$ with its induced path metric, denoted $\mathrm{dist}_{M\setminus N_s(K)}(\cdot, \cdot)$, for some constant $C$ depending only on $M$.
\end{prop}
\begin{proof}
    Let $x, y\in M\setminus N_s(K)$ be such that $\mathrm{dist}(r(x), r(y))\lesssim s$, where the implicit constant will be chosen later. Let $x=x_0, x_1,x_2,...,x_n=y$ be points in $M$ on the shortest path between $x$ and $y$ in $M\setminus N_s(K)$, such that
    \begin{enumerate}
        \item $[x_i, x_{i+1}]\cap N_s(K)=\emptyset$ for $i=0,1,...,n-1$, and
        \item $\sum_{i=0}^{n-1}\mathrm{dist}(x_i, x_{i+1})<2\mathrm{dist}_{M\setminus N_s(K)}(x, y)$.  
    \end{enumerate} 
    Note in particular that $\mathrm{dist}([x_i, x_{i+1}], r(x))\geq s$. We note that since $K$ is convex, we have $[r(x), r(y)]\subseteq K$. Since $r$ is the nearest-point retraction, we have $\angle([r(x), x], [r(x), r(y)]),\angle([r(y), y], [r(y), r(x)])\geq\pi/2$.
    \par Pick comparison triangles for $r(x)x_ix_{i+1}$ for $i=0,1,...,n-1$ and for $r(x)yr(y)$ in the 2-dimensional space $\mathbb{H}(-a^2)$ of constant curvature $-a^2$ (these exist by \cite{triangle-comp}). We glue them appropriately to get a hyperbolic polygon in $\mathbb{H}(-a^2)$. We work in the disk model, and we can suppose without loss of generality that $r(x)$ corresponds to the origin. Suppose that $x$, $y$ and $r(y)$ correspond to $A, B, C\in\mathbb{D}$, respectively. We know that
    \begin{enumerate}
        \item the angle between $[0,A]$ and $[0, C]$ is at least $\pi/2$, 
        \item the angle between $[C, 0]$ and $[C, B]$ is at least $\pi/2$, 
        \item $\mathrm{dist}(0, C)=\mathrm{dist}(r(x), r(y))$, $\mathrm{dist}(0, A), \mathrm{dist}(0, B)\geq s$, and 
        \item $\sum_{i=0}^{n-1}\mathrm{dist}([x_i, x_{i+1}])\geq \mathrm{dist}_{\mathbb{D}\setminus B(0, s)}(A, C)$ and hence $\mathrm{dist}_{M\setminus N_s(K)}(x, y)\gtrsim \mathrm{dist}_{\mathbb{D}\setminus B(0, s)}(A, B)$.
    \end{enumerate}
    \begin{claim}
        We have $\mathrm{dist}_{\mathbb{D}\setminus B(0, s)}(A, C)\gtrsim e^{as}\mathrm{dist}(0, C)$. 
    \end{claim}
    \begin{proof}
        We can rescale the metric so that $a=1$ and we are working in $\mathbb{H}^2(-1)=\mathbb{H}^2$. We can suppose without loss of generality that $\mathrm{dist}(0, A)=\mathrm{dist}(0, B)=s$, and that the angles $\angle([0, A], [0,C])=\angle([C,0], [C, B])=\frac{\pi}{2}$. We note that $\angle([0, A], [0,B])\approx \mathrm{dist}(0, C)$ for $\mathrm{dist}(0, C)$ small enough, and since the metric on $\mathbb{D}$ is $4\frac{dr^2+r^2d\theta^2}{(1-r^2)^2}$ in polar coordinates, we have 
        \begin{align*}
            \mathrm{dist}_{\mathbb{D}\setminus B(0,s)}(A, C)&\geq \frac{2\tanh \frac{s}{2}}{1-\tanh^2\frac{s}{2}}\angle([0, A], [0, B])\\
            &\approx \mathrm{dist}(0, C)\sinh s\approx e^s \mathrm{dist}(0, C).
        \end{align*}
    \end{proof}
    From the Claim we have 
    \begin{align*}
        \mathrm{dist}_{M\setminus N_s(K)}(x, y)&\gtrsim \mathrm{dist}_{\mathbb{D}\setminus B(0,s)}(A, B)\gtrsim e^{as}\mathrm{dist}(0, C)=e^{as}\mathrm{dist}(r(x), r(y)).
    \end{align*}
    It follows that $\mathrm{dist}(r(x),r(y))\lesssim e^{-as}\mathrm{dist}_{M\setminus N_s(K)}(x, y)$. 
\end{proof}
Applying Lemma \ref{lm:benoist-hulin} to the nearest-point retraction gives the following corollary.
\begin{cor}\label{cor:mollified-r-props}
    For any convex subset $K\subseteq M$ there exists a map $\tilde{r}:M\to K$ with $\sup_{x\in M}\mathrm{dist}(r(x), \tilde{r}(x))<\infty$ so that 
    \begin{gather*}
        \norm{\nabla\tilde{r}}\lesssim e^{-a\mathrm{dist}(\cdot, K)},\\
        \norm{H(\tilde{r})(X, X)}\lesssim e^{-a\mathrm{dist}(\cdot, K)}\norm{X}^2,
    \end{gather*}
    for any vector field $X$.
\end{cor}
\begin{remark}\label{remark:hessian-nabla-bounds-indep-of-k}
    Note that the constants in this corollary depend only on $M$, and are in particular independent of $K$. 
\end{remark}
\section{Reduction to an integral estimate of Green's function}
\label{section:harmonic-finite-shells}
In this section we show that, assuming $x\to\int_{U} G(x,y)d\mathrm{vol}(y)$ is bounded for a fixed large neighborhood $U$ of $K$, there exists a bounded subharmonic map $\Phi:M\to\mathbb{R}$ with the bound on the Laplacian 
\begin{align*}
    \Delta\Phi\geq e^{-a\mathrm{dist}(\cdot, K)}.
\end{align*} 
To construct $\Phi$ on $U$, we use the assumption on the integral of Green's function. On $M\setminus U$, we construct $\Phi$ as some function of the distance $\delta(x)=\mathrm{dist}(x, \tilde{r}(x))$. Hence to bound $\Delta\Phi$, we need bounds on the Laplacian and derivative of $\delta$. This essentially follows from the work of Benoist and Hulin \cite[Remark 4.6]{benoist-hulin-convex}, but we include a different proof in \S\ref{subsection:props-distance} as Proposition \ref{prop:laplacian-distance} for completeness. We then finish the construction of $\Phi$ in \S\ref{subsection:proof-greens-estimate-implies-soln}. 
\par Suppose now we are given a bounded subharmonic function $\Phi:M\to\mathbb{R}$ with $\Delta\Phi\geq e^{-a\mathrm{dist}(\cdot, K)}$. Suppose we are given a function $f:M\to N$ with $\norm{\tau( f)}\lesssim e^{-a\mathrm{dist}(\cdot, K)}$. On any ball $B(x, R)$, we can construct a harmonic map $h_R:B(x,R)\to N$ with $h_R=f$ on $\partial B(x, R)$. An estimate by Schoen and Yau from \cite{schoen-yau-laplacian-distance} shows that 
\begin{align*}
    \Delta\mathrm{dist}(f,g)\geq -\norm{\tau(f)}-\norm{\tau(g)},
\end{align*}
for any smooth functions $f,g:M\to N$. In fact we use a general formula from which Schoen and Yau derive this estimate in \S\ref{subsection:props-distance} to bound the Laplacian of $\delta(x)=\mathrm{dist}(x, \tilde{r}(x))$. 
\par Using this formula we see that $\Delta(\mathrm{dist}(h_R, f)+C\Phi)>0$ for some large constant $C$. Therefore $\mathrm{dist}(h_R, f)+C\Phi\leq C\sup{\abs{\Phi}}$ by the maximum principle. It follows that $\mathrm{dist}(h_R, f)$ is bounded uniformly in $R$. A classical argument combining Cheng's lemma, Schauder elliptic estimates and Arzela-Ascoli theorem then shows that we can take the limit of such harmonic maps as $R\to\infty$, to get a harmonic map $h_\infty:M\to N$ at a bounded distance from $f$. The details of this argument are in \S\ref{subsection:from-phi-to-harmonic-map}.
\par This part of the argument is similar to something that appears in the work of Donnelly \cite{donelly}. Specifically, it is shown in \cite[Lemma 3.1]{donelly} that given a function $f:M\to N$ and a bounded non-negative map $\Phi:M\to\mathbb{R}$ with $\Delta\Phi>\norm{\tau(f)}$, there exists a harmonic map $h:M\to N$ such that \[\mathrm{dist}(h(x), f(x))\leq\sup_x \Phi(x).\] In \cite{donelly}, $\Phi$ is constructed using an assumption on the integral of the Green's function, in a way completely analogous to how we construct $\Phi$ on a large neighborhood of $K$. However here we are able to construct $\Phi$ far away from $K$ without any assumptions on $K$ or the Green's function, by Proposition \ref{prop:construct-harmonic}.
\subsection{Properties of the distance function}\label{subsection:props-distance}
We set $\delta(x)=\mathrm{dist}(x, \tilde{r}(x))$. We will consider $\delta$ on $M\setminus N_C(K)$, for some $C$ large to be chosen later. In particular, we let $C$ be large enough so that $x\neq \tilde{r}(x)$ for all $x\in M\setminus N_C(K)$.
\begin{prop}\label{prop:laplacian-distance}
    For some $C>0$ large enough, the distance function $\delta$ is smooth on $M\setminus N_C(K)$ and has 
    \begin{align*}
        \Delta\delta\gtrsim 1\text{ and }\norm{\nabla\delta}\lesssim 1.
    \end{align*}
\end{prop}
\begin{proof}
    It is well-known that $\mathrm{dist}(\cdot, \cdot):M^2\to\mathbb{R}$ is smooth away from the diagonal, so the smoothness of $\delta$ follows from that of $\tilde{r}$.
    \par Since $\norm{\nabla\tilde{r}}$ is finite, $\tilde{r}$ is Lipschitz, and hence so is $\delta(x)=\mathrm{dist}(x, \tilde{r}(x))$. Therefore $\norm{\nabla\delta}<\infty$. 
    \par The rest of this proof is estimating $\Delta\delta$. We compare $\delta$ near some arbitrary point $x_0$ to the function $\mathrm{dist}(\cdot, \tilde{r}(x_0))$, using standard estimates on the Hessian of the distance function. The computation is analogous to that of Schoen and Yau \cite{schoen-yau-laplacian-distance}.
    \par We fix a large $C>0$ such that $\tilde{r}(x)\neq x$ for $x\in M\setminus N_{C}(K)$. Fix a point $x_0\in M\setminus N_C(K)$ and set $p_0=\tilde{r}(x_0)$. We first state a calculus claim, proved by a straightforward computation that we omit.
    \begin{claim}\label{claim:lap-composition}
        Let $f:A\to B$ be a smooth map between Riemannian manifolds, and let $g:B\to\mathbb{R}$ be a smooth function. Then for $x\in A$, 
        \begin{align*}
            \Delta_x(g\circ f)=dg_{f(x)}(\tau(f)_x)+\mathrm{tr}\left((d_xf)^* H(g)_{f(x)}\right).
        \end{align*}
    \end{claim}
    \begin{remark}
        We note that here $H(g)_{f(x)}$ is a symmetric bilinear form on $T_{f(x)}B$, that we pull back by $d_x f:T_x A\to T_{f(x)}B$ to get a symmetric bilinear form on $T_x A$. The trace then refers to the Riemannian metric on $A$. 
    \end{remark}
    We apply Claim \ref{claim:lap-composition} to the maps $(\mathrm{id}_X, \tilde{r}):M\setminus N_C(K)\to M^2$ and $\mathrm{dist}:M^2\to\mathbb{R}$. This yields 
    \begin{gather}\label{eq:lap-delta}
        \Delta_{x_0}\delta=\mathrm{dist}_*(\tau(\tilde{r})_{x_0})+\sum_\alpha H(\mathrm{dist})_{(x_0, p_0)}(e_\alpha\oplus\tilde{r}_*e_\alpha, e_\alpha\oplus\tilde{r}_*e_\alpha),
    \end{gather}
    where $e_\alpha$ is some orthonormal basis for $T_{x_0}M$. Note that $e_\alpha\oplus \tilde{r}_*e_\alpha\in T_{x_0}M\oplus T_{p_0}M\cong T_{(x_0, p_0)}(M^2)$. Applying Claim \ref{claim:lap-composition} once again to the map $\mathrm{id}_X\times \mathrm{const}_{p_0}:M\to M\times\{p_0\}\subset M\times M$ (here by $\mathrm{const}_{p_0}$ we denote the constant map $M\to \{p_0\}\subset M$), we get 
    \begin{align}\label{eq:lap-dist-p0}
        \Delta_{x_0} \mathrm{dist}(\cdot, p_0)=\sum_\alpha H(\mathrm{dist})_{(x_0, p_0)}(e_\alpha\oplus 0, e_\alpha\oplus 0).
    \end{align} 
    Note that from \cite[Lemma 2.3]{benoist-hulin-pinched-hadamard}, we see that \begin{align}\label{eq:hessian-bound}
        \norm{H(\mathrm{dist})_{(x_0, p_0)}}\leq b\coth(bC).
    \end{align}
    Here the norm is defined relative to the Riemannian metric on $T_{x_0}M\oplus T_{p_0}M$ by $\norm{H}=\sup_{\norm{X}=1} \abs{H(X, X)}$.  Subtracting (\ref{eq:lap-delta}) from (\ref{eq:lap-dist-p0}), using (\ref{eq:hessian-bound}) with the fact that $\norm{\tilde{r}_*e_\alpha}\lesssim e^{-aC}$, we see that 
    \begin{align*}
        \abs{\Delta_{x_0}\delta - \Delta_{x_0} \mathrm{dist}(
            p_0, \cdot
        )}\lesssim e^{-aC}(1+b^2\coth^2(bC))\to 0\text{ as }C\to\infty.
    \end{align*}
    It is well-known that $\Delta_{x_0}\mathrm{dist}(p_0, \cdot)\geq a$ (see e.g. \cite[Lemma 2.5]{Benoist2017HarmonicQM}), so for $C>0$ large enough, we have $\Delta_{x_0}\delta>a/2$ for $x_0\in M\setminus N_C(K)$.
\end{proof}
\subsection{Constructing bounded subharmonic functions}\label{subsection:proof-greens-estimate-implies-soln}
We use the following Proposition to construct $\phi$ on $N_{d+1}(K)\setminus N_d(K)$.
\begin{prop}\label{prop:construct-harmonic}
    Let $S\subseteq\partial_\infty M$ and let $K=\mathrm{CH}(S)$. For all $d>d_0$, where $d_0=d_0(M)$ is some constant depending only on $M$, there exists a bounded subharmonic function $\phi_d:M\to\mathbb{R}$ such that 
    \begin{align*}
        \Delta\phi_d\geq 1\text{ on }N_{d+1}(K)\setminus N_d(K),
    \end{align*}
    so that $\sup_x \abs{\phi_d(x)}$ does not depend on $d$.
\end{prop}
\begin{proof}
Let $f:[0,\infty)\to[0,\infty)$ be a $C^2$ function. Then 
\begin{align*}
    \Delta (f\circ\delta)=f'(\delta)\Delta \delta+f''(\delta)\norm{\nabla\delta}^2.
\end{align*}
By Proposition \ref{prop:laplacian-distance} we can suppose $\Delta\delta\geq A$ and $\norm{\nabla\delta}^2\leq B$, for some positive constants $A,B$. We will construct $\phi_d$ as $f\circ\delta$, for a suitable function $f$.
\begin{claim} There exists a $C^1$ function $u:\mathbb{R}\to[0,\infty)$ be a $C^1$ function with the following properties:
    \begin{enumerate}
        \item We have $u=0$ on $\left(-\infty,-\frac{1}{2}\right]$ and $u=1$ on $\left[0, 1\right]$.
        \item On $(-\infty, 1]$, $u$ is non-decreasing, and on $[1,\infty)$, $u$ is decreasing.
        \item We have $u\approx e^{-\varepsilon x}$ for $x$ large enough, for some $\varepsilon>0$.
        \item We have $Au+B\min(u', 0)\geq 0$.
    \end{enumerate}
\end{claim}
\begin{proof} Fix $\varepsilon<\frac{A}{B}$. Let $v:[-1/2,\infty)\to\mathbb{R}$ be a $C^1$ function with the following properties:
    \begin{enumerate}
        \item For some small $\lambda>0$, we have 
        \begin{align*}
            v(x)=v_\text{lower}(x)=-\frac{1}{(2x-1)^2}\text{ for }x\in\left(-\frac{1}{2}, -\frac{1}{2}+\lambda\right]
        \end{align*}            
        \item In the interval $\left(-\frac{1}{2}+\lambda, 2\right)$, the function is defined by $v(x)=v_\mathrm{middle}(x)$, where
        \begin{align*}
            v_\text{middle}'(x)=\left\lbrace \begin{matrix}
                \frac{1-x}{\frac{3}{2}-\varepsilon}\cdot v_\text{lower}'(-1/2+\lambda) & \text{for }-\frac{1}{2}+\lambda<x\leq 1 \\
                -\varepsilon(x-1) & \text{for }1<x<2 
            \end{matrix}\right.
        \end{align*}
        and $v_\text{middle}(-1/2+\lambda)=v_\mathrm{lower}(-1/2+\lambda)$.
        \item For $x\geq 2$, we have $v(x)=v_\mathrm{middle}(2)-\varepsilon (x-2)$.
    \end{enumerate}
    Then it is immediate that the function 
    \begin{align*}
        u=\left\lbrace\begin{matrix} e^v & \text{for }x>-\frac{1}{2}\\ 0 & \text{for }x\leq-\frac{1}{2}\end{matrix}\right.
    \end{align*}
    has properties 1), 2), 3). Note that when $x\geq 1$, we have 
    \begin{align*}
        v'(x)=\max(-\varepsilon, \varepsilon(1-x))\geq -\varepsilon,
    \end{align*} 
    and hence $\frac{u'}{u}\geq-\varepsilon>-\frac{A}{B}$. Therefore $Bu'+Au>0$ when $x\geq 1$. When $x<1$, we have $B\min(u', 0)+Au=Au>0$, so 4) is shown.
\end{proof}
We set $f(x)=\int_0^x u(t-d)dt$. By exponential decay of $u$ at infinity, $f$ is bounded. When $d\geq C$ from Proposition \ref{prop:laplacian-distance}, we have 
\begin{align*}
    \Delta(f\circ\delta)\geq A u(\delta-d)+B\min(u'(\delta-d), 0)\geq 0,
\end{align*}
so $f\circ\delta$ is subharmonic, and moreover $\Delta(f\circ\delta)=\Delta\delta\gtrsim 1$ whenever $d\leq \delta(x)\leq d+1$, or equivalently $x\in N_{d+1}(K)\setminus N_d(K)$. We rescale $f\circ\delta$ to get $\phi_d$.
\end{proof}
The following lemma uses the assumption on the integral of the Green's function to construct $\Phi$ on $N_C(K)$ for $C$ large enough, and hence finishes the construction of $\Phi$.
\begin{lm}\label{lm:greens-estimate-implies-soln}
    Let $S\subseteq\partial_\infty M$ and let $K=\mathrm{CH}(S)$. Then if \[\sup_x\int_{N_{d_0(M)+2}(K)} G(x,y)d\mathrm{vol}(y)<\infty,\] there exists a bounded map $\Phi:M\to\mathbb{R}$ with $\Delta\Phi\geq e^{-a\mathrm{dist}(\cdot, K)}$.
\end{lm}
\begin{proof}
Let $\chi:M\to [0,1]$ be a smooth map with $\chi=1$ on $N_{d_0+1}(K)$ and $\chi=0$ on $M\setminus N_{d_0+2}(K)$. Then we construct $\Phi$ as 
\begin{align*}
    \Phi(x)=\sum_{n=\lceil d_0\rceil}^\infty e^{-an} \phi_n-\int_{M} \chi(y)G(x,y)d\mathrm{vol}(y). 
\end{align*}
We have $\int_M \chi(y)G(x,y)d\mathrm{vol}(y)\leq\int_{N_{d_0+2}(K)} G(x,y)d\mathrm{vol}(y)$ which is bounded by assumption, so $\Phi$ is bounded. Note that 
\begin{align*}
    -\Delta\int_M \chi(y)G(x,y)d\mathrm{vol}(y)=\chi(x)\geq\left\lbrace\begin{matrix}
        1 & \text{on }N_{d_0+1}(K),\\
        0 & \text{on }M.
    \end{matrix}\right.
\end{align*}
Therefore $\Delta\Phi\gtrsim e^{-a\mathrm{dist}(\cdot, K)}$, so after rescaling we can take $\Delta\Phi\geq e^{-a\mathrm{dist}(\cdot, K)}$.
\end{proof}
\subsection{Using bounded subharmonic functions to finish the proof}\label{subsection:from-phi-to-harmonic-map}
We derive all our Theorems from the following corollary of Lemma \ref{lm:greens-estimate-implies-soln}.
\begin{cor}\label{cor:finish-from-green}
    Let $S\subseteq\partial_\infty M$ and let $K=\mathrm{CH}(S)$. Let $f:M\to N$ be a smooth map between pinched Hadamard manifolds such that
        \begin{align*}
            \norm{\tau(f)}\lesssim e^{-a\mathrm{dist}(\cdot, K)}.
        \end{align*}
    There exists a constant $C=C(M)>0$ such that if \[\sup_x\int_{N_C(K)} G(x,y)d\mathrm{vol}(y)<\infty,\] there exists a harmonic map $h:M\to N$ at a bounded distance from $f$.
\end{cor}
\begin{proof}
    Set $C=d_0(M)+2$, and let $\Phi:M\to\mathbb{R}$ be the function from Lemma \ref{lm:greens-estimate-implies-soln}. We fix an arbitrary $x_0\in M$. For all $d>0$, we let $h_d:B(x_0, d)\to N$ be the harmonic map such that $h_d=f$ on $\partial B(x_0,d)$. Then by \cite{schoen-yau-laplacian-distance}, we have 
    \begin{align*}
        \Delta\mathrm{dist}(h_d, f)\geq -\norm{\tau(f)}\gtrsim -e^{-a\mathrm{dist}(\cdot, K)}.
    \end{align*}
    Hence for a suitable constant $C'>0$ that does not depend on $d$, we have 
    \begin{align*}
    \Delta\left(\mathrm{dist}(h_d, f)+C'\Phi\right)>0.
    \end{align*}
    By the maximum principle, we have 
    \begin{align*}
        \sup_{B(x_0,d)}\mathrm{dist}(h_d, f)\leq 2C'\sup_{M}\abs{\Phi}=:D.
    \end{align*}
    \par For any fixed $x\in M$, for arbitrarily large $n$, $h_n$ maps the ball $B(x, 2)$ to the fixed bounded set $N_{D}(f(B(x, 2)))$. As $h_n$ is harmonic, by Cheng's lemma (see \cite{cheng}), we have \[\sup_n \sup_{y\in B(x, 1)} \norm{\nabla_y h_n}<\infty,\] for any $x\in M$. It follows by the Arzela-Ascoli theorem that there is a sequence $k_n\to\infty$ such that $h_{k_n}\to h_\infty$ as $n\to\infty$, uniformly on compact sets. By the classical elliptic estimates (that can be found in \cite[Theorem 70, p. 303]{petersen})
    \begin{align*}
        \norm{h_n}_{C^{2,\alpha}(B(x,\varepsilon))}&\lesssim \norm{h_{n}}_{C^\alpha(B(x, \varepsilon))},
    \end{align*}
    for $\alpha<1$, so by Arzela-Ascoli applied again, there exists a further subsequence, that we also denote $k_n$, such that $H(h_{k_n})\to H(h_\infty)$ uniformly on compact sets. Therefore $h_\infty$ is also harmonic, is defined everywhere and \begin{align*}\mathrm{dist}(h_\infty(x), f(x))\leq\limsup_{n\to\infty}\mathrm{dist}(h_{k_n}(x), f(x))\leq D\end{align*}
    for any $x\in M$. 
\end{proof}
\begin{remark}
    Note that the assumption that $x\to\int_{N_C(K)}G(x,y)d\mathrm{vol}(y)$ is bounded is only used to construct $\Phi$ on $N_C(K)$. The only place we use dimension bounds on $S$ is to verify this assumption. In particular, if there exists a bounded subharmonic map $\phi:M\to\mathbb{R}$ with $\Delta\phi\geq 1$ on $N_C(K)$, then there exists a harmonic map $h:M\to M$ at a bounded distance from the nearest-point retraction $r:M\to\mathrm{CH}(S)=K$ (with no assumptions on $\overline{\dim}S$).
\end{remark}
\section{Upper bound on the volume of the convex hull within a large ball}\label{section:growth-of-cnvx-hull}
In this section we show that given an upper bound on the invariant upper Minkowski dimension of a set $S\subseteq\partial_\infty M$, we get an upper bound on $\mathrm{vol}(B(x, \rho)\cap N_d(\mathrm{CH}(S)))$ for any $d>0$. \par Recall that the invariant upper Minkowski dimension is defined using the visual metric $\mathrm{dist}_x^\mathrm{vis}(\cdot, \cdot)$, where $x\in M$ is some fixed basepoint, satisfying 
    \begin{align*}
        A^{-1}e^{-a\mathrm{dist}(x, [y,z])}\leq \mathrm{dist}_x^{\mathrm{vis}}(y, z)\leq A e^{-a\mathrm{dist}(x, [y,z])},
    \end{align*}
    for all $y, z\in\partial_\infty M$.
\begin{lm}\label{lm:low-dim-exp-growth}
    Let $S$ be a set in the boundary $\partial_\infty M$. Then for all $x\in M$, we have 
    \begin{align*}
        \mathrm{vol}(B(x, \rho)\cap N_d(\mathrm{CH}(S)))\lesssim e^{a\rho \beta},
    \end{align*}
    for any $\beta>\overline{\dim}S, d>0$, where the implicit constant depends only on $M, d$ and $\beta$.
\end{lm}
We first outline the proof of Lemma \ref{lm:low-dim-exp-growth}. We will estimate the volume of the intersection of $N_d(\mathrm{CH}(S))$ with annuli $\mathrm{An}(R)=B(x, R+1)\setminus B(x, R)$. To achieve this, we cover the set $S$ with balls $B_1, B_2,...,B_N$ of radius $e^{-aR}$ in the visual metric $d_x^\mathrm{vis}$. The proof has three ingredients, sketched below. 
\begin{enumerate}
    \item We first show that $\mathrm{CH}(S)\subseteq N_{C'}(\mathrm{Cone}(x,S))$ for some absolute constant $C'>0$, where $\mathrm{Cone}(x, S)=\bigcup_y [x,y]$. 
    \item We next show that $N_{C}(\mathrm{Cone}(x, B_i))\setminus B(x, R)\subseteq \mathrm{Cone}(x,\tilde{B}_i)$, where $\tilde{B}_i$ is the ball with the same center as $B_i$, but has radius larger by a bounded factor. This result explicitly uses that $B_i$ has radius $e^{-aR}$.
    \item Finally, we show that $\mathrm{Cone}(x, \tilde{B}_i)\cap \mathrm{An}(R)$ has bounded diameter independent of $R$, and hence bounded volume.
\end{enumerate}
Combining these three ingredients, we see that $\mathrm{vol}(B(x, \rho)\cap N_d(\mathrm{CH}(S)))\lesssim N=N(R)\lesssim e^{a\rho \overline{\dim}S}$ by assumption. Uniformity follows from the fact that $\beta$ is required to be strictly larger than the invariant upper Minkowski dimension. The rest of this section is devoted to proving Lemma \ref{lm:low-dim-exp-growth}.
\subsection{Notation} For $S\subseteq\partial_\infty M$, denote by $\mathrm{Cone}(x, S)$ the union of geodesic rays with one endpoint $x$ and the other endpoint (at infinity) in $S$. We denote by $\pi_x:M\setminus\{x\}\to\partial_\infty M$ the projection that maps $y\in M\setminus\{x\}$ to the unique point $z\in\partial_\infty M$ so that $y\in[x, z]$. We also write, for the duration of this proof $\mathrm{An}(R)=B(x, R+1)\setminus B(x, R)$. We also remind the reader that $[a,b]$ denotes the geodesic segment connecting $a,b\in\overline{M}$ (potentially infinite on one or both sides).
\subsection{Esitmating the convex hull with the cone}
The purpose of this subsection is to show the proposition below.
 \begin{prop}\label{prop:ch-subset-cone}
    There exists a constant $C$ such that for all $S\subseteq\partial_\infty M$ and $x\in M$, we have $\mathrm{CH}(S)\subseteq N_C(\mathrm{Cone}(x, S))$. 
\end{prop}
\begin{proof}
        Denote by $\mathrm{GH}(S)$ the union of all geodesics with both endpoints in $S$. Clearly $\mathrm{GH}(S)\subseteq\mathrm{CH}(S)$.
        \begin{claim} For some constant $C'$ depending only on $M$, we have $\mathrm{CH}(S)\subseteq N_{C'}(\mathrm{GH}(S))$. \end{claim}
        \begin{proof} Suppose not, so that we have a sequence $S_n$ of subsets of $\partial_\infty M$ with points $x_n\in\mathrm{CH}(S_n)$ such that $\mathrm{dist}(x_n, \mathrm{GH}(S_n))\to\infty$. Since the action of $\mathrm{Isom}(M)$ on $M$ has cobounded orbits, let $\Phi$ be a compact subset of $M$ that intersects every orbit. Without loss of generality, we can modify $S_n, x_n$ by an isometry so that $x_n\in\Phi$, for each $n$. We pass to a subsequence of $x_n$ such that $x_n\to x$. Then $\mathrm{dist}(x, \mathrm{GH}(S_n))\to\infty$ and $\mathrm{dist}(x, \mathrm{CH}(S_n))\to 0$, as $n\to\infty$. We equip $\partial_\infty M$ with the visual metric based at $x$. Since $\mathrm{dist}(x, \mathrm{GH}(S_n))\to\infty$, we have 
            \begin{align*}
                \sup_{y, z\in S_n} d_x^\mathrm{vis}(y,z)\approx e^{-a \mathrm{dist}(x, [y, z])}\to 0
            \end{align*}
            as $n\to\infty$. But then $\mathrm{diam}(S_n)\to 0$ and hence $\mathrm{dist}(x, \mathrm{CH}(S_n))\to \infty$, which is a contradiction.
        \end{proof}
        Note that for any $a, b\in S$, since $M$ is $\delta$-hyperbolic (as a metric space) for some $\delta>0$, we have $[a, b]\subseteq N_\delta([x, a]\cup [x, b])$, and therefore 
       \begin{align*}
           \mathrm{GH}(S)\subseteq N_\delta(\mathrm{Cone}(x, S)),
       \end{align*}
       for all $x\in M$ and $S\subseteq \partial_\infty M$. In particular $\mathrm{CH}(S)\subseteq N_{C'+\delta}(\mathrm{Cone}(x, S))$, so we set $C=C'+\delta$. 
    \end{proof}
    \subsection{Estimating neighborhood of a cone}
    In this subsection, we show the following proposition.
    \begin{prop}\label{prop:cone-nbhd}
        For any $R>0$ and constant $C>0$, there exists a constant $\tilde{C}=\tilde{C}(C, M)$, such that for all sets $S\subseteq\partial_\infty M$, we have 
       \begin{align*}
           N_C(\mathrm{Cone}(x, S))\setminus B(x, R)\subseteq \mathrm{Cone}(x, N_{\tilde{C} e^{-aR}}(S)).
       \end{align*}
    \end{prop}
    \begin{proof}
       This is essentially equivalent to the following claim. We remind the reader that $\pi_x(y)$ is the unique point of intersection of the half-ray $xy$ with the boundary at infinity $\partial_\infty M$.
    \begin{claim}\label{claim:diff-dist}
        For any $C>0$, there exists $D=D(C)$ such that for all $x,y,z\in M$ with $\mathrm{dist}(y, z)\leq C$, we have 
        \begin{align*}
            \mathrm{dist}(x, y)-\mathrm{dist}(x, [\pi_x(y),\pi_x(z)])\leq D(C).
        \end{align*}
    \end{claim}
    \begin{proof}
        Suppose there exist sequences $x_n, y_n, z_n$ with $\mathrm{dist}(y_n, z_n)\leq C$ and 
        \begin{align*}
            \mathrm{dist}(x_n, y_n)-\mathrm{dist}(x_n,[\pi_{x_n}(y_n), \pi_{x_n}(z_n)]) \to\infty\text{ as }n\to\infty.
        \end{align*}
        Let $w_n$ be the nearest point on $[\pi_{x_n}(y_n), \pi_{x_n}(z_n)]$ to $x_n$. 
        After applying an appropriate isometry of $M$, we can suppose $y_n$ lies in a fixed compact set for all $n$. Since $\mathrm{dist}(y_n, z_n)\leq C$, all $z_n$ also lie in a fixed compact set. Therefore we can pass to a subsequence so that $y_n\to y\in M$ and $z_n\to z\in M$. 
        \par  Note that \[\mathrm{dist}(x_n, y_n)-\mathrm{dist}(x_n,[\pi_{x_n}(y_n), \pi_{x_n}(z_n)])\leq \min(\mathrm{dist}(x_n, y_n),\mathrm{dist}(y_n, w_n))\] so the sequences $(x_n)$ and $(w_n)$ converge to some points on the boundary at infinity $\partial_\infty M$. Denote these points $x$ and $w$, respectively. Note that $\pi_{x_n}(y_n)$ converges to the unique point $\pi_x(y)\in\partial_\infty M$ such that $y\in [x, \pi_x(y)]$. We define $\pi_x(z)$ analogously, and observe that $\pi_{x_n}(z_n)\to\pi_x(z)$. Since $w_n$ lies on the geodesic  $[\pi_{x_n}(y_n), \pi_{x_n}(z_n)]$, in the limit we have $w=\pi_x(y)$ or $w=\pi_x(z)$. The claim below then implies that $\pi_x(y)=\pi_x(z)$, i.e. $x, y, z$ all lie on the same geodesic.
        \begin{claim}\label{claim:minima-conv}
            Let $\alpha,\beta,\gamma$ be distinct points in $\partial_\infty M$. Suppose that $\alpha_n, \beta_n, \gamma_n$ are sequences of points in $M$ that converge to $\alpha,\beta,\gamma$, respectively. If $\omega_n$ is the nearest point on $[\beta_n,\gamma_n]$ to $\alpha_n$, then the set $\{\omega_n:n=1,2,...\}$ is bounded.
        \end{claim}
        \begin{proof}
            This follows from the basic properties of horofunctions \cite[Chapter II.8]{Bridson1999MetricSO}, but we sketch the proof for completeness. Let $f_n:(-L_n, R_n)\to M$ be the arc-length parameterization of $[\beta_n, \gamma_n]$, appropriately shifted so that $f_n\to f$, where $f:\mathbb{R}\to M$ is the arc-length parameterization of $[\beta,\gamma]$. We fix an arbitrary basepoint $o\in M$, and consider the functions $d_n:M\to\mathbb{R}$ given by \[d_n(x)=\mathrm{dist}(x, \alpha_n)-\mathrm{dist}(\alpha_n, o).\]
            Then $d_n\to d_\infty$ uniformly on compact sets. In particular, $d_n\circ f_n\to d_\infty\circ f$ uniformly on compact sets. These are all convex functions, and $d_\infty\circ f$ has a unique minimum, so the minima of $d_n\circ f_n$ remain bounded.  
        \end{proof}
        Since $w_n$ are unbounded, by Claim \ref{claim:minima-conv} we see that $\pi_x(y)=\pi_x(z)=w$. Therefore $x,y,z$ lie on the same geodesic. For ease of notation, we assume that $y$ lies between $x$ and $z$.
        \par Denote the angle at $y_n$ between $[y_n, x_n]$ and $[y_n, w_n]$ by $\theta_n$. Write $A_n=\mathrm{dist}(y_n, x_n)$, $B_n=\mathrm{dist}(y_n, w_n)$ and $C_n=\mathrm{dist}(x_n, w_n)$. By the arguments above we see that $\theta_n\to\pi$ as $n\to\infty$. Since $M$ is a CAT$(-a^2)$ metric space, we see that for a triangle with sides $aA_n, aB_n, aC_n$ in $\mathbb{H}^2$, the corresponding angle opposite $C_n$ is at least $\theta_n$. By the hyperbolic law of cosines
        \begin{align*}
            \cosh aC_n\geq\cosh aB_n\cosh aA_n+(-\cos\theta_n)\sinh aB_n\sinh aA_n.
        \end{align*}
        Therefore $\frac{\cosh aC_n}{\cosh a(A_n+B_n)}\to 1$ as $n\to\infty$ and hence $B_n+(A_n-C_n)\to 0$. However by assumption $A_n-C_n\to\infty$, and we get a contradiction since $B_n\geq 0$.
    \end{proof}
    Let $z\in N_C(\mathrm{Cone}(x, S))\setminus B(x, R)$, so that for some $y\in \mathrm{Cone}(x, S)$ we have $\mathrm{dist}(y, z)\leq C$. By Claim \ref{claim:diff-dist}, 
    \begin{align*}\mathrm{dist}(x, [\pi_x(y), \pi_x(z)])\geq \mathrm{dist}(x, y)-D,
    \end{align*}
    so that $\mathrm{dist}_x^\text{vis}(\pi_x(y), \pi_x(z))\leq Ae^{aD}e^{-a\mathrm{dist}(x, y)}$. But $\pi_x(y)\in S$ since $y\in\mathrm{Cone}(x, S)$, and hence $\mathrm{dist}_x^\text{vis}(\pi_x(z), S)\leq Ae^{aD}e^{-a\mathrm{dist}(x, y)}$. Since $\mathrm{dist}(x, z)\geq R$, we have $\mathrm{dist}(x, y)\geq R-\mathrm{dist}(y, z)\geq R-C$ and hence $\mathrm{dist}_x^\text{vis}(\pi_x(z), S)\leq Ae^{a(C+D)}e^{-aR}$. We thus let $\tilde{C}=Ae^{a(C+D(C))}$, and see that $\pi_x(z)\in N_{\tilde{C}}(S)$, and hence $z\in\mathrm{Cone}(x, N_{\tilde{C}}(S))$. Since $\tilde{C}$ or this argument do not depend on the choice of $z$, we are done.
    \end{proof}
    \subsection{Decomposition} 
    In this and the next two subsections, we show Lemma \ref{lm:low-dim-exp-growth} using Propositions \ref{prop:ch-subset-cone} and \ref{prop:cone-nbhd}. 
    \par Set $\varepsilon=e^{-aR}$, and cover the set $S$ by $N=N(\varepsilon)$ balls of radius $\varepsilon$, centered at $y_1, y_2, ..., y_N$. Since $\mathrm{CH}(S)\subseteq N_C(\mathrm{Cone}(x, S))$ by Proposition \ref{prop:ch-subset-cone}, and $S\subseteq\bigcup_{i=1}^N B(y_i, \varepsilon)$, we have 
    \begin{align}
        N_d(\mathrm{CH}(S))\cap\mathrm{An}(R)&\subseteq \mathrm{An}(R)\cap\bigcup_{i=1}^N N_{C+d}(\mathrm{Cone}(x, B(y_i, \varepsilon)))\setminus B(x, R)\nonumber\\
        &\subseteq \mathrm{An}(R)\cap\bigcup_{i=1}^N \mathrm{Cone}(x, B(y_i, \varepsilon+\tilde{C}e^{-aR})),\label{eq:decomp}
    \end{align}
    where we used Proposition \ref{prop:cone-nbhd} in going from the first to the second line. 
    \subsection{Volume bound on cones over visual balls}
    In this subsection, we show that each piece in the decomposition (\ref{eq:decomp}) has bounded volume.
    \begin{claim}\label{claim:bdd-diam}
        Fix a constant $C$. Then for points $x\in M, y\in\partial_\infty M$, define the set 
        \begin{align*}
            S_{R,C}(x, y)=\mathrm{Cone}\left(x, \{z\in\partial_\infty M: \mathrm{dist}(x, [y, z])\geq R\}\right)\cap B(x, R+C)\setminus B(x, R). 
        \end{align*}
        Then the diameter of $S_{R,C}$ is bounded by a constant $D(C)$ depending only on $C$ and $M$.
    \end{claim}
    \begin{proof}
        Suppose the diameter of $S_{R,C}(x,y)$ is unbounded. Then there exist sequences $x_n, w_n\in M, y_n, z_n\in \partial_\infty M$ and $R_n>0$ such that 
        \begin{gather*}
            \mathrm{dist}(x_n, [y_n, z_n])\geq R_n,\\ w_n\in[x_n, z_n],\\ R_n\leq\mathrm{dist}(x_n, w_n)\leq R_n+C,\\ \mathrm{dist}(w_n, p_n)\to\infty,
        \end{gather*}
    where $p_n$ is the point on $[x_n, y_n]$ at a distance $R_n$ from $x_n$. By applying an isometry of $M$, we can suppose that $p_n$ is in a fixed compact set. Pass to a subsequence so that $p_n\to p\in M$. Note that the diameter of $S_{R,C}(x,y)$ is at most $2(R+C)$, so in particular $R_n\to\infty$. Hence $\mathrm{dist}(x_n, p_n)=R_n\to\infty$, so we have $x_n\to x\in\partial_\infty M$ along some subsequence. Then $y_n\to y$ with $p\in [y, x]$ and pass to a further subsequence so that $z_n\to z\in\partial_\infty M$ and $w_n\to w\in\overline{M}$.\par Note that since $[y_n, z_n]$ is disjoint from the ball centered at $x_n$ through $p_n$, it follows that $[y, z]$ is disjoint from the horoball $H$ based at $x$ passing through $p$. In particular $y, z\neq x$. We also have $w\in [x, z]\cap N_C(H)\setminus H$ which is a bounded set in $M$, so $\mathrm{dist}(p, w)<\infty$. This is a contradiction. 
    \end{proof}
    Denote by $V(C)$ the maximal volume of a ball of radius $D(C)$. Note that for $z\in B(y_i, \varepsilon+\tilde{C}e^{-aR})$, we have 
    \begin{align*}
        \varepsilon+\tilde{C}e^{-aR}=e^{-aR}(1+\tilde{C})\geq\mathrm{dist}_x^\mathrm{vis}(y_i, z)\geq A^{-1}e^{-a\mathrm{dist}(x, [y_i,z])},
    \end{align*}
    and hence $\mathrm{dist}(x, [y_i, z])\geq R-\frac{1}{a}\log A(1+\tilde{C})$. Therefore   
    \begin{align*}
        B(x, R+1)\cap \mathrm{Cone}(x, B(y_i, \varepsilon(1+AC')))\subseteq S_{R-\frac{1}{a}\log A(1+\tilde{C}), 1+\frac{1}{a}\log A(1+\tilde{C})}(x, y_i),
    \end{align*}
    and hence 
    \begin{align*}
        \mathrm{vol}(\mathrm{An}(R)\cap N_{C+d}(\mathrm{Cone}(x, B(y_i, \varepsilon))))\leq V(1+a^{-1}\log A(1+\tilde{C}))=:V_0(d).
    \end{align*}
    \subsection{Finishing the proof} We now combine previous results to show the main volume estimate. 
    We have
    \begin{align*}
        \mathrm{vol}(N_d(\mathrm{CH}(S))\cap \mathrm{An}(R))&\leq N(e^{-aR}) V_0\lesssim e^{aR\beta} V_0.
    \end{align*}
    Hence we have 
    \begin{align*}
        \mathrm{vol}(N_d(\mathrm{CH}(S))\cap B(x,\rho))&\leq \sum_{r=0}^{\lfloor\rho\rfloor} \mathrm{vol}(\mathrm{CH}(K)\cap\mathrm{An}(r))\lesssim e^{a\rho \beta} V_0(d),
    \end{align*}
    where the implicit constant depends only on $M$ and $\beta$.
\section{Proof of the main results}\label{section:proof}
\subsection{Estimates on the heat kernel}\label{subsection:heat-kernel-estimates}
We collect some estimates on the heat kernel in pinched Hadamard manifolds and hyperbolic spaces we will use to bound $\int_{N_C(K)}G(x,y)d\mathrm{vol}(y)$. 
\par Recall that $H(x, y, t)$ denotes the heat kernel on $M$, for distinct $x,y\in M$ and $t\geq 0$. The connection to Green's function is through the identity 
\begin{align*}
    G(x,y)=\int_0^\infty H(x,y,t)dt.
\end{align*}
\begin{prop}\label{prop:kernel-estimates}
    Assume $t\geq 1$, and denote $\rho=\mathrm{dist}(x, y)$. 
    \begin{enumerate}
        \item On a pinched Hadamard manifold with sectional curvature at most $-a^2$, we have 
        \begin{align*}
            H(x, y, t)\lesssim (1+\rho^n)e^{-\frac{\rho^2}{4t}-\frac{(n-1)^2a^2}{4}t}.
        \end{align*}
        \item On $\mathbb{H}^n$, we have 
        \begin{align*}
            H(x, y, t)\lesssim (1+\rho^n)e^{-\frac{\rho^2}{4t}-\frac{(n-1)^2}{4}t-\frac{n-1}{2}\rho}.
        \end{align*}
    \end{enumerate}
\end{prop}
\begin{proof}
    \begin{enumerate}
        \item 
        Let $\lambda_1(M)$ be the bottom of the spectrum of the negative Laplacian $-\Delta$ on $M$. Note that $\Delta$ is elliptic, so in particular $\lambda_1(M)\geq 0$. We have the bound derived by Davies in \cite{davies-estimate}
    \begin{align*}
        H(x, y, t)\lesssim \frac{1}{\sqrt{\mathrm{vol}(B(x,r))\mathrm{vol}(B(y,r))}} e^{-\frac{\rho^2}{4t}-\lambda_1(M)t}=\overline{H}(\rho, t),
    \end{align*}
    where $r=\min\left(1, \sqrt{t}, \frac{t}{\rho}\right)$. Since $t\geq 1$, we have $r=\min\left(1, \frac{t}{\rho}\right)$. Note that since $r\leq 1$, we have $\mathrm{vol}(B(x, r))\approx r^n$ by Bishop's volume estimates. Therefore 
    \begin{align*}
        \overline{H}(\rho, t)&\approx r^{-n} e^{-\frac{\rho^2}{4t}-\lambda_1(M)t}=\max\left(1, \frac{\rho}{t}\right)^n e^{-\frac{\rho^2}{4t}-\lambda_1(M)t}\\ &\approx \left(\frac{\rho+t}{t}\right)^ne^{-\frac{\rho^2}{4t}-\lambda_1(M)t}\lesssim (1+\rho^n)e^{-\frac{\rho^2}{4t}-\frac{(n-1)^2a^2}{4}t},
    \end{align*}
    where in the last inequality we used $\lambda_1(M)\geq \frac{(n-1)^2a^2}{4}$, as shown by McKean in \cite{lower-bound-laplacian}. 
    \item By \cite[Theorem 3.1]{heat-kernel-hyp-space}, we have 
    \begin{align*}
        H(x, y, t)&\approx \frac{(1+\rho)(1+\rho+t)^\frac{n-3}{2}}{t^\frac{n}{2}}e^{-\frac{\rho^2}{4t}-\frac{(n-1)^2}{4}t-\frac{n-1}{2}\rho}\\
        &\approx \frac{1+\rho}{(2+\rho)^\frac{3}{2}} \left(\frac{1+\rho}{t}+1\right)^\frac{n}{2}e^{-\frac{\rho^2}{4t}-\frac{(n-1)^2}{4}t-\frac{n-1}{2}\rho}\\
        &\lesssim (2+\rho)^\frac{n-1}{2}e^{-\frac{\rho^2}{4t}-\frac{(n-1)^2}{4}t-\frac{n-1}{2}\rho}\\
        &\lesssim (1+\rho^n)e^{-\frac{\rho^2}{4t}-\frac{(n-1)^2}{4}t-\frac{n-1}{2}\rho}.
    \end{align*}
    \end{enumerate} 
\end{proof}
\subsection{Proof of Theorem \ref{thm:main-hyp-space-full}}\label{subsection:main-hyp-space-full}
Let ${r}:\mathbb{H}^n\to\mathrm{CH}(S)=K$ be the nearest-point retraction, and let $\tilde{r}:\mathbb{H}^n\to\mathbb{H}^n$ be the smooth map from Corollary \ref{cor:mollified-r-props}.
\begin{claim}
    There exists $\varepsilon=\varepsilon(d, r)>0$, so that 
    \begin{align*}
        \int_{N_d(K)} H(x, y, t)d\mathrm{vol}(y)\lesssim e^{-\varepsilon t},
    \end{align*}
    uniformly in $x$.
\end{claim}
\begin{proof} 
    By Lemma \ref{lm:low-dim-exp-growth} there exists some $\beta<n-1$ with 
\begin{align*}
\mathrm{vol}(B(x, \rho)\cap K)\lesssim e^{\beta \rho},
\end{align*}
uniformly in $x\in\mathbb{H}^n$.
    \par We have by Proposition \ref{prop:kernel-estimates},
    \begin{align*}
        \int_{N_d(K)} H(x, y, t)d\mathrm{vol}(y)\lesssim\sum_{\rho=1}^\infty (1+\rho^n)e^{-\frac{\rho^2}{4t}-\frac{(n-1)^2}{4}t-\frac{n-1}{2}\rho} \mathrm{vol}(\mathrm{An}(\rho)\cap N_d(K)),
    \end{align*}
    where $\mathrm{An}(\rho)=B(x, \rho)\setminus B(x, \rho-1)$. Since 
    \begin{align*}\mathrm{vol}(\mathrm{An}(\rho)\cap N_d(K))\leq\mathrm{vol}(B(x, \rho)\cap N_d(K))\lesssim e^{\beta\rho},\end{align*}
     we have 
    \begin{align*}
        \int_{N_d(K)} H(x, y, t)d\mathrm{vol}(y)\lesssim \sum_{\rho=1}^\infty (1+\rho^n)e^{-\frac{(n-1)^2}{4}t}\exp\left(-\frac{\rho^2+2t(n-1-2\beta)\rho}{4t}\right)\\
        \lesssim e^{-\frac{(n-1)^2-(n-1-2\beta)^2}{4}t}\sum_{\rho=1}^\infty (1+\rho^n) \exp\left(-\frac{(\rho+(n-1-2\beta)t)^2}{4t}\right) \\
        \lesssim e^{-\beta (n-1-\beta) t}\int_1^\infty (1+\rho^n) \exp\left(-\frac{(\rho+(n-1-2\beta)t)^2}{4t}\right)d\rho.
    \end{align*}
    The final integral grows at most polynomially in $t$, so the claim holds for any $0<\varepsilon<\beta(n-1-\beta)$.
\end{proof}
It follows that 
\begin{align*}
     \int_{N_d(K)} \int_0^\infty H(x,y,t)\norm{\tau(\tilde{r})(y)} dt d\mathrm{vol}(y)\lesssim 1,
\end{align*}
uniformly in $x$, as $\norm{\tau(\tilde{r})}\lesssim 1$. Since $G(x,y)=\int_0^\infty H(x,y,t)dt$, we are done by Corollary \ref{cor:finish-from-green}.
\subsection{Proof of Theorem \ref{thm:main-hyp-space}}\label{subseciton:main-hyp-space}
    Let $S$ be the image of $\partial\iota:S^1\to\partial\mathbb{H}^n$, and let $K$ be the convex hull of $S$. 
    \begin{claim}
        For some $d$ large enough, there exists a Lipschitz map $f:N_d(K)\to\mathbb{H}^2$ such that 
        \begin{align*}
            \sup_{x\in\mathbb{H}^2}\mathrm{dist}(f\circ\iota(x), x)<\infty.
        \end{align*}
    \end{claim}
    \begin{proof}
        Note that $\mathbb{H}^2=\bigcup_{z\in S^1}[-z, z]$, so that 
        \begin{align*}
            \iota(\mathbb{H}^2)=\bigcup_{z\in S^1}\iota([-z, z]).
        \end{align*}
        Each $\iota([-z,z])$ is a quasigeodesic with the same constants, so Morse lemma implies that $\iota([-z, z])\subseteq N_C([\iota(-z), \iota(z)])$ for some $C>0$ that depends only on quasi-isometry constants of $\iota$. Therefore $\iota(\mathbb{H}^2)\subseteq N_C(K)$. Similarly we have $[\iota(z), \iota(w)]\subseteq N_{C}([z, w])$, so $\mathrm{GH}(S)\subseteq N_C(\iota(\mathbb{H}^2))$. We have already shown as part of the first Claim of Lemma \ref{lm:low-dim-exp-growth} that $K\subseteq N_{C'}(\mathrm{GH}(S))$ for some constant $C'>0$ that depends only on $n$, so that $K\subseteq N_{C+C'}(\iota(\mathbb{H}^2))$, and therefore the quasi-isometric embedding $\iota:\mathbb{H}^2\to N_C(K)$ is quasisurjective.
        \par Therefore there exists a quasi-inverse $\tilde{f}:N_{C}(K)\to \mathbb{H}^2$, for all $C$ large enough, meaning $\sup_{x\in \mathbb{H}^2}\mathrm{dist}(x, \tilde{f}\circ\iota(x))<\infty$. We in fact construct a quasi-inverse on a larger set $\tilde{f}:N_{C+1}(K)\to\mathbb{H}^2$, so that by the construction of Benoist and Hulin from \cite[Proposition 2.4]{benoist-hulin-pinched-hadamard} we can construct a Lipschitz map $f:N_C(K)\to\mathbb{H}^2$ so that $\sup_{x}\mathrm{dist}(\tilde{f}(x), f(x))<\infty$. Then $f$ is a Lipschitz quasi-inverse of $\iota$, as claimed.
    \end{proof}
    Now let $r:\mathbb{H}^n\to K$ be the nearest-point retraction, and write $\hat{r}=f\circ {r}$, for some $f$ as in the Claim. Then $\hat{r}:\mathbb{H}^n\to\mathbb{H}^2$ is Lipschitz with 
    \begin{align*}
        \mathrm{Lip}\left(\hat{r}|_{\mathbb{H}^n\setminus N_d(K)}\right)\lesssim e^{-ad},
    \end{align*}
    for all $d>0$, by Proposition \ref{prop:nearest-point-lipschitz}. By Lemma \ref{lm:benoist-hulin} there exists a smooth map $\tilde{r}:\mathbb{H}^n\to\mathbb{H}^2$ so that 
    \begin{align*}
        \sup_{x\in\mathbb{H}^n}\mathrm{dist}(\tilde{r}&(x), \hat{r}(x))<\infty,\\
        \norm{\nabla_x \tilde{r}}&\lesssim e^{-a\mathrm{dist}(x, K)},\\
        \norm{H(\tilde{r})_x(X, X)}&\lesssim e^{-a\mathrm{dist}(x, K)} \norm{X}^2,
    \end{align*}
    for all $x\in\mathbb{H}^n, X\in T_x\mathbb{H}^n$. It follows from the first inequality and the Claim that 
    \begin{align*}
        \sup_{x\in\mathbb{H}^2} \mathrm{dist}(x, \tilde{r}\circ\iota(x))<\infty.
    \end{align*}
    Note that each set in $\{\gamma S:\gamma\in\mathrm{Isom}(\mathbb{H}^n)\}$ is a quasicircle with the same quasisymmetry constants as $S$, so by \cite[Theorem 18, Lemma 16]{gehring} we see that 
    \begin{align*}
        \overline{\dim}^{\mathrm{Isom}(\mathbb{H}^n)} S<n-1.
    \end{align*}
    Therefore we can apply the Claim from \S\ref{subsection:main-hyp-space-full} to get
    \begin{align*}
         \int_{N_d(K)} \int_0^\infty H(x,y,t)\norm{\tau(\tilde{r})(y)} dt d\mathrm{vol}(y)\lesssim 1,
    \end{align*}
    uniformly in $x$, as $\norm{\tau(\tilde{r})}\lesssim 1$. Since $G(x,y)=\int_0^\infty H(x,y,t)dt$, by Corollary \ref{cor:finish-from-green}, we are done.
\subsection{Proof of Theorem \ref{thm:main-hadamard}}\label{subsection:main-hadamard}
    As always, let $K$ be the convex hull of $S$.
    \begin{claim}
        There exists $\varepsilon=\varepsilon(K, d)$, so that 
        \begin{align*}
            \int_{N_d(K)} H(x,y,t)d\mathrm{vol}(y)\lesssim e^{-\varepsilon t},
        \end{align*}
        uniformly in $x$.
    \end{claim}
    \begin{proof}
        Let $\overline{\dim} S<\beta<\frac{n-1}{2}$. Note that by Lemma \ref{lm:low-dim-exp-growth} we have 
        \begin{align*}
            \mathrm{vol}(B(x, \rho)\cap N_d(K))\lesssim e^{a\rho\beta},
        \end{align*}
        uniformly in $x$. Therefore by Proposition \ref{prop:kernel-estimates}, we have 
        \begin{align*}
            \int_{N_d(K)} H(x, y, t)d\mathrm{vol}(y)&\lesssim\sum_{\rho=1}^\infty (1+\rho^n) e^{-\frac{\rho^2}{4t}-\frac{(n-1)^2a^2}{4t}} \mathrm{vol}(\mathrm{An}(\rho)\cap N_d(K))\\
            &\leq \sum_{\rho=1}^\infty (1+\rho^n) e^{-\frac{\rho^2}{4t}-\frac{(n-1)^2a^2}{4t}} \mathrm{vol}(B(x,\rho)\cap N_d(K))\\
            &\lesssim \sum_{\rho=1}^\infty (1+\rho^n) e^{-\frac{\rho^2-4ta\beta\rho}{4t}-\frac{(n-1)^2a^2}{4}t}\\
            &=e^{-a^2t\left(\frac{(n-1)^2}{4}-\beta^2\right)}\sum_{\rho=1}^\infty (1+\rho^n) \exp\left(-\frac{(\rho-2ta\beta)^2}{4t}\right)\\
            &\lesssim e^{-a^2t\left(\frac{(n-1)^2}{4}-\beta^2\right)}\int_1^\infty(1+\rho^n) \exp\left(-\frac{(\rho-2ta\beta)^2}{4t}\right).
        \end{align*}
        The integral in the final line grows at most polynomially, so the Claim holds for any $0<\varepsilon<a^2 \left(\frac{(n-1)^2}{4}-\beta^2\right)$.
    \end{proof}
    Let $r:M\to K$ be the nearest-point retraction, and let $\tilde{r}:M\to M$ be as in Corollary \ref{cor:mollified-r-props}. Fix an arbitrary $d>0$. We have
    \begin{align*}
        \int_{N_d(K)}\int_0^\infty  H(x,y,t)\norm{\tau(\tilde{r})(y)}dt d\mathrm{vol}(y)\lesssim 1,
    \end{align*}
    uniformly in $x$, as $\norm{\tau(\tilde{r})}\lesssim 1$. Since $G(x,y)=\int_0^\infty H(x,y,t)dt$, by Corollary \ref{cor:finish-from-green}, we are done.
    
\bibliographystyle{amsplain}
\bibliography{harmonic-projection}
\end{document}